\documentclass[12pt]{amsart}

\usepackage{amsfonts,a4,latexsym}
\usepackage{amsmath}                      
\usepackage{amsthm,amscd,amssymb}         
\usepackage{dsfont}
\usepackage{graphicx}
\usepackage{epsfig}
\usepackage[latin1]{inputenc}
\usepackage[T1]{fontenc}

\newtheorem{theorem}{Theorem}[section]
\newtheorem{lemma}[theorem]{Lemma}

\theoremstyle{definition}

\theoremstyle{remark}

\numberwithin{equation}{section}

\newcommand{\rr}{{\mathbb R}}
\newcommand{\rd}{{\mathbb R^d}}

\newcommand{\nat}{{\mathbb N}}

\newcommand{\Exp}{{\mathbb E}}
\newcommand{\Ker}{\operatorname{Ker}}

\newcommand{\eqd}{\stackrel{\rm d}{=}}

\newcommand{\Prob}{\mathbb P}
\allowdisplaybreaks

\begin{document}

\sloppy
\title[Hausdorff dimension of the graph]{Hausdorff dimension of the graph of an operator semistable L\'evy process} 

\author{Lina Wedrich}
\address{Lina Wedrich, Mathematisches Institut, Heinrich-Heine-Universit\"at D\"usseldorf, Universit\"atsstr. 1, D-40225 D\"usseldorf, Germany}
\email{lina.wedrich\@@{}uni-duesseldorf.de} 

\date{\today}

\begin{abstract}
Let $X=\{X(t):t\geq0\}$ be an operator semistable L\'evy process in $\rd$ with exponent $E$, where $E$ is an invertible linear operator on $\rd$. For an arbitrary Borel set $B\subseteq\rr_+$ we interpret the graph $Gr_X(B)=\{(t,X(t)):t\in B\}$ as a semi-selfsimilar process on $\rr^{d+1}$, whose distribution is not full, and calculate the Hausdorff dimension of $Gr_X(B)$ in terms of the real parts of the eigenvalues of the exponent $E$ and the Hausdorff dimension of $B$. We use similar methods as applied in \cite{MX} and \cite{KW}.
\end{abstract}

\thanks{This work has been supported by Deutsche Forschungsgemeinschaft (DFG) under grant KE1741/6-1}

\maketitle

\baselineskip=18pt


\section{Introduction}

Let $X=(X(t))_{t\geq0}$ be a L\'evy process in $\rd$. Namely, $X$ is a stochastically continuous process with c\`adl\`ag paths that has stationary and independent increments and starts in $X(0)=0$ almost surely. The distribution of $X$ is uniquely determined by the distribution of $X(1)$ which can be an arbitrary infinitely divisible distribution. The process $X$ is called $(c^E,c)$-operator semistable, if the distribution of $X(1)$ is full, i.e. not supported on any lower dimensional hyperplane, and there exists a linear operator $E$ on $\rd$ such that
\begin{equation}\label{opsem}
\left\{X(ct)\right\}_{t\geq0}\stackrel{\text{fd}}{=} \left\{c^E X(t)\right\}_{t\geq0}\quad \text{ for some } c>1.
\end{equation}
Here $\stackrel{\text{fd}}{=}$ denotes equality of all finite dimensional distributions and 
$$c^E:=\sum_{n=0}^\infty \frac{(\log c)^n}{n!}E^n.$$
If for some $\alpha\in(0,2]$ the exponent $E$ is a multiple of the identity, i.e. $E=\alpha\cdot I$, we call the process $(c^{1/\alpha},c)$-semistable. The L\'evy process is called operator stable if \eqref{opsem} holds for all $c>0$.

The aim of this paper is to calculate the Hausdorff dimension $\dim_H Gr_X(B)$ of the graph $Gr_X(B)=\{(t,X(t)): t\in B\}$ of an operator semistable L\'evy process $X=(X(t))_{t\geq0}$ for an arbitrary Borel set $B\subseteq\rr_+$.

For an arbitrary subset $F$ of $\rd$ the $s$-dimensional Hausdorff measure $\mathcal{H}^s(F)$ is defined as
\begin{align*}
	\mathcal{H}^s(F)=\lim_{\delta\rightarrow0}\inf\left\{\sum_{i=1}^{\infty}|F|_i^s: |F_i|\leq\delta \mbox{ and } F\subseteq\bigcup_{i=1}^\infty F_i\right\},
\end{align*}
where $|F|=\sup\{\|x-y\|:x,y\in F\}$ denotes the diameter of a set $F\subseteq\mathbb{R}^d$ and $\|\cdot\|$ is the Euclidean norm. It can be shown that the value $\dim_{H} F=\inf\left\{s:\mathcal{H}^s(F)=0\right\}=\sup\left\{s:\mathcal{H}^s(F)=\infty\right\}$ exists and is unique for all subsets $F\subseteq\rd$. The critical value $\dim_{H} F$ is called the Hausdorff dimension of $F$. Further details on the Hausdorff dimension can be found in \cite{Fal} and \cite{Mat}.

In the past efforts have been made to generate dimension results for L\'evy processes, which fulfill certain scaling properties. An overview can for example be found in \cite{KX} or \cite{Xiao}. For an operator semistable L\'evy process $X$ and an arbitrary Borel set $B\subseteq\rr_+$ Kern and Wedrich \cite{KW} calculated the Hausdorff dimension of the range $\dim_{H}X(B)$ in terms of the real parts of the eigenvalues of the exponent $E$ and the Hausdorff dimension of $B$. The result is a generalization of the one stated in Meerschaert and Xiao \cite{MX}, who calculated the Hausdorff dimension $\dim_{ H} X(B)$ for an operator stable L\'evy process. For an arbitrary operator semistable L\'evy process $X$  our aim is to generalize the methods used to prove the results above by interpreting the graph $Gr_X(B)=\{(t,X(t)):t\in B\}$ as a process on $\rr^{d+1}$, which fulfills the scaling property \eqref{opsem} for a certain exponent but whose distribution is not full. The most prominent example of a semistable, non-stable distribution is perhaps the limit distribution of the cumulative gains in a series of St. Petersburg games. In this particular case, Kern and Wedrich \cite{KW2} already calculated the Hausdorff dimension $\dim_{H}Gr_X([0,1])$ of the corresponding graph over the interval $[0,1]$ employing the method described above. 
Furthermore, in the case that $X$ is a dilation stable L\'evy process on $\rd$, i.e. an operator stable L\'evy process with a diagonal exponent, Xiao and Lin \cite{LX} and Hou \cite{Hou} calculated the Hausdorff dimension $\dim_{H}Gr_X(B)$ for an arbitrary Borel set $B\subseteq\rr_+$

This paper is structured as follows: In Section 2.1 we recall spectral decomposition results from \cite{MS}, which enable us to decompose the exponent $E$ and thereby the operator semistable L\'evy process $X$ according to the distinct real parts of the eigenvalues of $E$. Section 2.2 contains certain uniformity and positivity results from \cite{KW} for the density functions of the process $X$, which will be helpful in the proofs of our main results. The main results on the Hausdorff dimension of the graph of an operator semistable L\'evy process are stated and proven in Section 3.

Throughout this paper $K$ denotes an unspecified positive and finite constant that can vary in each occurrence. Fixed constants will be denoted by $K_1, K_2$, etc.

\section{Preliminaries}

\subsection{Spectral decomposition}

Let $X$ be a $(c^E,c)$-operator semistable L\'evy process. Factor the minimal polynomial of $E$ into $q_1(x)\cdot\ldots\cdot q_p(x)$ where all roots of $q_i$ have roots with real parts equal to $a_i$ and $a_i<a_j$ for $i<j$. Let $\alpha_j = a_j^{-1}$ so that $\alpha_1>\ldots>\alpha_p$, and note that $0<\alpha_j\leq2$ by Theorem 7.1.10 in \cite{MS}. Define $V_j=\Ker(q_j(E))$. According to Theorem 2.1.14 in \cite{MS} $V_1\oplus\cdots\oplus V_p$ is then a direct sum decomposition of $\rd$ into $E$ invariant subspaces. In an appropriate basis, $E$ is then block-diagonal and we may write $E=E_1\oplus\cdots\oplus E_p$ where $E_j:V_j\rightarrow V_j$ and every eigenvalue of $E_j$ has real part equal to $a_j$. Especially, every $V_j$ is an $E_j$-invariant subspace of dimension $d_j= \dim V_j$ and $d=d_1+\ldots+d_p$. Write $X(t)=X^{(1)}(t)+\ldots+X^{(p)}(t)$ with respect to this direct sum decomposition, where by Lemma 7.1.17 in \cite{MS}, $\{X^{(j)}(t), t\geq0\}$ is a $(c^E_j,c)$-operator semistable L\'evy process on $V_j$. We can now choose an inner product $\langle\cdot,\cdot\rangle$ on $\rd$ such that the $V_j, j\in\{1, \ldots,p\}$, are mutually orthogonal and throughout this paper we will let $\|x\|^2=\langle x, x\rangle$  be the associated Euclidean norm. In particular we have for $t=c^r m >0$ that
\begin{align}
\|X(t)\|^2\eqd\|c^{rE}X(m)\|^2=\|c^{rE_1}X^{(1)}(m)\|^2+\ldots+\|c^{rE_p}X^{(p)}(m)\|^2,
\end{align}
with $r\in\mathbb{Z}$ and $m\in[1,c)$. 

The following lemma states a result on the growth behavior of the exponential operators $t^{E_j}$ near the origin $t=0$. It is a variation of Lemma 2.1 in \cite{MX} and a direct consequence of Corollary 2.2.5 in \cite{MS}.

\begin{lemma}\label{specbound}
	For every $j\in\{1,\ldots,p\}$ and every $\epsilon>0$ there exists a finite constant $K\geq1$ such that for all $0<t\leq1$ we have
	\begin{align}
	K^{-1}t^{a_j+\epsilon}\leq\|t^{E_j}\|\leq K t^{a_j-\epsilon}
	\end{align}
	and
	\begin{align}
	K^{-1}t^{-(a_j-\epsilon)}\leq\|t^{-E_j}\|
	\leq K t^{-(a_j+\epsilon)}.
	\end{align}
\end{lemma}

Throughout this paper we will denote by $\alpha_j=1/a_j$ the reciprocals of the real parts of the eigenvalues of the exponent $E$ with $0<\alpha_p<\ldots<\alpha_1\leq2$.

\subsection{Properties of the density function}

The following three lemmas state uniformity results of operator semistable L\'evy processes. They will be very helpful in the proofs of our main theorems. The lemmas are taken from Kern and Wedrich \cite{KW}. Let $X=\{X(t)\}_{t\geq0}$ be an operator semistable L\'evy process on $\rd$ and $g_t, t> 0$, the corresponding continuous density functions. Lemma 2.2 in \cite{KW} states the following:

\begin{lemma}\label{uniconfun}
	The mapping $(t,x)\mapsto g_t(x)$ is continuous on $(0,\infty)\times \rd$ and we have
	\begin{align}
		\sup_{t\in[1,c)}\sup_{x\in\rd} |g_t(x)|<\infty.
	\end{align}
\end{lemma}

As a consequence we get a result on the existence of negative moments of an operator semistable L\'evy process $X=\{X(t)\}_{t\geq0}$ on $\rd$ given in Lemma 2.3 of \cite{KW}.

\begin{lemma}\label{uniexp}
	For any $\delta\in(0,d)$ we have
	\begin{align}
	\sup_{t\in[1,c)} \Exp[\|X(t)\|^{-\delta}]<\infty.
	\end{align}
\end{lemma}

Furthermore, we will need a uniform positivity result for the density functions taken from Lemma 2.4 of \cite{KW}.

\begin{lemma}\label{denspos}
	Let $\{X(t)\}_{t\geq0}$ be an operator semistable L\'evy process with $\alpha_1>1$, $d_1=1$ and with density $g_t$ as above. Then there exist constants $K>0$, $r>0$ and uniformly bounded Borel sets $J_t\subseteq \rr^{d-1}\cong V_2\oplus\cdots\oplus V_p$ for $t\in[1,c)$ such that
	\begin{align}
	g_t(x_1,\ldots,x_p)\geq K>0 \text{ for all } (x_1,\ldots,x_p)\in[-r,r]\times J_t.
	\end{align}
	Further, we can choose $\{J_t\}_{t\in[1,c)}$ such that $\lambda^{d-1}(J_t)\geq R$ for every $t\in[1,c)$. Note that the constants $K$, $r$ and $R$ do not depend on $t\in[1,c)$.
\end{lemma}


\section{Main results}

The following two Theorems are the main results of this paper. The constants $\alpha_1, \alpha_2$ and $d_1$ are defined as in Section 2.1 by means of the spectral decompostition.

\begin{theorem}\label{mainresult1}
	Let $X=\{X(t), t\in\rr_+\}$ be an operator semistable L\'evy process on $\rd$ with $d\geq2$. Then for any Borel set $B\subseteq\rr_+$ we have almost surely
	\begin{align*}
	\dim_H Gr_X(B) = \left\{
	\begin{array}{ll} \dim_H B \cdot \max(\alpha_1,1), & \text{ if } \alpha_1\dim_H B \leq d_1, \\
	1+\max(\alpha_2, 1)\cdot(\dim_H B - \frac{1}{\alpha_1}), & \text{ if } \alpha_1\dim_H B > d_1.\end{array}
	\right.
	\end{align*}
	
\end{theorem}


The dimension result for the one-dimensional case reads as follows:

\begin{theorem}\label{mainresult2}
	Let $X=\{X(t), t\in\rr_+\}$ be a ($c^{1/\alpha},c$)-semistable L\'evy process on $\rr$. Then for any Borel set $B\subseteq\rr_+$ we have almost surely
	$$
	\dim_HGr_X(B) = \left\{
	\begin{array}{ll} \dim_H B \cdot \max(\alpha,1), & \text{ if } \alpha\dim_H B \leq 1, \\
	1+\dim_H B - \frac{1}{\alpha}, & \text{ if } \alpha\dim_H B > 1.\end{array}
	\right.
	$$
\end{theorem}


Let $X=(X(t))_{t\geq0}$ be a $(c, c^E)$-operator semistable L\'evy process on $\rd$ and let $\alpha_1>\ldots>\alpha_p$ denote the reciprocals of the real parts of the eigenvalues of $E$ as defined in Section 2.1. We want to calculate the Hausdorff dimension of the graph $Gr_X(B)$ of $X$ for an arbitrary Borel set $B\subseteq\rr_{+}$. Therefore, we define the process $Z = (Z(t))_{t\geq0}$ as $Z(t)=(t, X(t))$ for all $t\geq0$. This gives us $\dim_H Z(B) = \dim Gr_X(B)$. One can easily see that $Z$ is also a L\'evy process and fulfills the scaling property of a $(c,c^F)$-operator semistable process where
$$
F = \left( \begin{array}{cc}
1 & 0 \\
0 & E \end{array} \right).
$$
Nevertheless, the process $Z$ itself is not operator semistable in the sense of the definition given in the Introduction as the distribution of $Z(1)$ is obviously not full.

As mentioned in the Introduction, the Hausdorff dimension $\dim_H X(B)$ of the range of an operator semistable L\'evy process $X$ has already been calculated in \cite{KW} as
\begin{equation}\label{DimensionRange}
\dim_{\rm H}X(B)=
\begin{cases} 
\alpha_1\dim_{\rm H}B & \text{if } \alpha_1\dim_{\rm H}B \leq d_1, \\
1+\alpha_2\left(\dim_{\rm H}B-\frac1{\alpha_1}\right) & \text{if } \alpha_1\dim_{\rm H}B > d_1,
\end{cases}
\end{equation}
almost surely for $d\geq 2$.
Hence, for the reasons mentioned above, we are now able to use the parts of the result \eqref{DimensionRange} and the corresponding proofs where fullness of the process was not required. All other parts, however, have to be calculated anew.

The proof of Theorem 3.1 is split into two parts. First we will obtain the upper bounds for $\dim_HGr_X(B)$ by choosing a suitable sequence of coverings. This method goes back to Pruitt and Taylor \cite{PT} and Hendricks \cite{H}. Afterwards we will use standard capacity arguments in order to prove the lower bounds.

\subsection{Upper Bounds}

For a L\'evy process $\{X(t)\}_{t\geq0}$ let
\begin{align}
T_X(a,s)=\int_0^s 1_{B(0,a)}(X(t))dt
\end{align}
be the sojourn time in the closed ball $B(0,a)$ with radius $a$ centered at the origin up to time $s>0$.

 The following covering lemma is due to Pruitt and Taylor \cite[Lemma 6.1]{PT}
 
 \begin{lemma}\label{PT}
 	Let $Z=\{Z(t)\}_{t\geq0}$ be a  L\'evy process in $\rr^{d+1}$ and let $\Lambda(a)$ be a fixed $K_{1}$-nested family of cubes in $\rr^{d+1}$ of side $a$ with $0<a\leq1$. For any $u\geq0$ let $M_{u}(a,s)$ be the number of cubes in $\Lambda(a)$ hit by $Z(t)$ at some time $t\in [u,u+s]$. Then
 	$$\Exp\left[M_{u}(a,s)\right]\leq 2\,K_{1}s\cdot\left(\Exp\left[T_Z\left(\tfrac{a}{3},s\right)\right]\right)^{-1}.$$
 \end{lemma}
 
 In order to prove the upper bounds of Theorem \ref{mainresult1} we now need to calculate sharp lower bounds of the expected sojourn times $\Exp[T_Z(a,s)]$ of the graph $Z=\{(t,X(t)), t\geq0\}$ of an operator semistable L\'evy process on $\rd$.

In their paper Kern and Wedrich \cite[Theorem 2.6]{KW} calculated upper and lower bounds for the expected sojourn times $\Exp[T_X(a,s)]$ of an operator semistable L\'evy process:

\begin{theorem}\label{rangesojournbounds}
	Let $X=\{X(t)\}_{t\geq0}$ be as in Theorem 3.1. For any $0<\alpha_{2}''<\alpha_{2}<\alpha_{2}'<\alpha_{1}''<\alpha_{1}<\alpha_{1}'$ there exist positive and finite constants $K_{6},\ldots,K_{9}$ such that
	\begin{itemize}
		\item[(i)] if $\alpha_{1} \leq d_{1}$, then for all $0<a\leq 1$ and $a^{\alpha_{1}}\leq s\leq1$ we have
		\begin{equation*}
		K_{6}a^{\alpha_{1}'}\leq\Exp[T_X(a,s)]\leq K_{7}a^{\alpha_{1}''}.
		\end{equation*}
		\item[(ii)] if $\alpha_{1}>d_{1}=1$, for all $0<a\leq a_{0}$ with $a_0>0$ sufficiently small, and all $a^{\alpha_{2}}\leq s \leq 1$ we have
		\begin{equation*}
		K_{8}a^{\rho'}\leq\Exp[T_X(a,s)]\leq K_{9}a^{\rho''},
		\end{equation*}
		where $\rho''=1+\alpha_{2}''(1-\frac{1}{\alpha_{1}})$ and $\rho'=1+\alpha_{2}'(1-\frac{1}{\alpha_{1}})$.
	\end{itemize}
\end{theorem}

Similarly to the theorem above we will now calculate lower bounds for the expected sojourn times $\Exp[T_Z(a,s)]$ of the graph $Z=\{(t,X(t)), t\geq0\}$ of an operator semistable L\'evy process on $\rd$. The upper bounds can also be calculated but are not stated here as they are not needed to determine the Hausdorff dimension.

\begin{theorem}\label{ST}
	
	Let $Z=\{(t,X(t)), t\geq 0\}$, where $X=\{X(t),t\geq0\}$ is as in Theorem 3.1.
	\begin{itemize}
		\item[(i)] If $\alpha_1\leq d_1$ and $\alpha_1\geq1$, then for all $0<a\leq1$ and $a^{\alpha_1}\leq s\leq1$ and any $\alpha_1<\alpha_1'$ there exists a positive and finite constant $K_2$ such that
		\begin{align*}
		\Exp[T_Z(a,s)]\geq K_2 a^{\alpha_1'}.
		\end{align*}
		\item[(ii)] If $\alpha_1\leq d_1$ and $\alpha_1<1$, then for all $0<a\leq1$ and $a\leq s\leq1$ and any $\epsilon>0$ there exists a positive and finite constant $K_3$ such that
		\begin{align*}
		\Exp[T_Z(a,s)]\geq K_3 a^{1+\epsilon}.
		\end{align*}
		\item[(iii)] If $\alpha_1> d_1=1$ and $\alpha_2\geq1$, then for any $0<\alpha_2<\alpha_2'<\alpha_1$ and all $a>0$ small enough, say $0<a\leq a_0$, and all $a^{\alpha_2}\leq s\leq1$ there exists a positive and finite constant $K_4$ such that
		\begin{align*}
		\Exp[T_Z(a,s)]\geq K_4 a^{1+\alpha_2'(1-\frac{1}{\alpha_1})}.
		\end{align*}
		\item[(iv)] If $\alpha_1> d_1=1$ and $\alpha_2<1$, then for all $a>0$ small enough, say $0<a\leq a_0$, and all $\frac{a}{\sqrt{p+1}}\leq s\leq1$, there exists a positive and finite constant $K_5$ such that
		\begin{align*}
		\Exp[T_Z(a,s)]\geq K_5 a^{2-\frac{1}{\alpha_1}}.
		\end{align*}
	\end{itemize}
\end{theorem}

\begin{proof}

	(i) \& (ii) Let $\alpha_1'>\alpha_1$. Looking at the proof of Theorem \ref{rangesojournbounds} part (i) (i.e. Theorem 2.6 part (i) in \cite{KW}) one realizes that the fullness is not needed there. Hence we can use this result to prove part (i) and (ii) of the present theorem. In order to do so we need to further examine the exponent	
		$$
		F = \left( \begin{array}{cc}
		1 & 0 \\
		0 & E \end{array} \right)
		$$
	of the process $Z$. Analogously to Section 2.1 denote by $\tilde{\alpha}_1>\ldots>\tilde{\alpha}_q$ the reciprocals of the real parts of the eigenvalues of $F$ and by $\tilde{d}_1$ the dimension of the $F_1$ invariant subspace of $\rr^{d+1}$, where $F_1$ is (analagously to $E_1$) the blockmatrix, whose eigenvalues have real part equal to $\tilde{\alpha}_1^{-1}$. Furthermore, let $\tilde{\alpha}_1'$ be such that $\tilde{\alpha}_1'=\tilde{\alpha}_1+\alpha_1'-\alpha_1$.
	
	In part (i) we have that $\alpha_1\leq d_1$ and $\alpha_1\geq1$. Then $\tilde{\alpha_1}=\alpha_1$ and $\tilde{d}_1\geq d_1$. All together we have $\tilde{\alpha}_1\leq\tilde{d}_1$ and by Theorem \ref{rangesojournbounds} there now exists a positive constant $K_2$ such that
	\begin{align*}
	\Exp[T_Z(a,s)]\geq K_2 a^{\tilde{\alpha}_1'} = K_2 a^{\alpha_1'}
	\end{align*}
	for all $0<a\leq1$ and $a^{\alpha_1}\leq s\leq1$.
	
	On the other hand in part (ii) we have $\alpha_1<1$. Then $\tilde{\alpha}_1=1$ and $\tilde{d}_1=1$, so that again $\tilde{\alpha}_1\leq\tilde{d}_1$. For any $\epsilon>0$, by Theorem \ref{rangesojournbounds} there now exitsts a postive constant $K_3$ such that
	\begin{align*}
	\Exp[T_Z(a,s)]\geq K_3 a^{\tilde{\alpha}_1+\epsilon} = K_3 a^{1+\epsilon}
	\end{align*}
	for all $0<a\leq1$ and $a\leq s \leq1$.
	
	(iii) Let $0<\alpha_j<\alpha_j'<\alpha_{j-1}$ for all $j=2,\ldots,p$. Choose $i_0, i_1\in\nat_0$ such that $c^{-i_0}<a\leq c^{-i_0+1}$ and $c^{-i_1}<c^{-i_0 \alpha_2} \leq c^{-i_1+1}$. For $t\in(0,1]$ we can write $t=mc^{-i}$ with $m\in[1,c)$ and $i\in\nat_0$. By Lemma \ref{specbound} we then have
	\begin{equation}\begin{split}\label{ubXj}
	\|X^{(j)}(t)\| \stackrel{\rm d}{=} \|c^{-iE_j}X^{(j)}(m)\|\leq \|c^{-iE_j}\|\,\|X^{(j)}(m)\| \leq K\,c^{-i/{\alpha_{j}'}}\|X^{(j)}(c^it)\|
	\end{split}\end{equation}
	for all $j=1,\ldots,p$. Note that, since $d_1=1$, for $j=1$ in \eqref{ubXj} we can choose $K=1$ and $\alpha_1'=\alpha_1$. Furthermore, since $\alpha_2\geq1$ there exists a constant $a_0>0$ such that for all $0<a\leq a_0$ we have $a^{\alpha_2}\leq \frac{a}{\sqrt{p+1}}$. Altogether, for $0<a\leq a_0$ this gives us 
	\begin{align*}
	\Exp[T_Z(a,s)] & = \int_0^s \Prob\left(\|Z(t)\|<a\right)dt = \int_0^s \Prob\left(\|(t,X(t))\|<a\right)dt\\
	&\geq\int_0^s \Prob\left(|X^{(1)}(t)|<\frac{a}{\sqrt{p+1}}, \|X^{(j)}(t)\|<\frac{a}{\sqrt{p+1}},2\leq j\leq p, |t|<\frac{a}{\sqrt{p+1}}\right)dt\\
	&\geq \int_0^{a^{\alpha_2}} \Prob\left(|X^{(1)}(t)|<\frac{a}{\sqrt{p+1}},\|X^{(j)}(t)\|<\frac{a}{\sqrt{p+1}},2\leq j\leq p\right)dt\\
	&\geq \int_0^{c^{-i_1}} \Prob\left(|X^{(1)}(t)|<\frac{a}{\sqrt{p+1}},\|X^{(j)}(t)\|<\frac{a}{\sqrt{p+1}},2\leq j\leq p\right)dt\\
	& = \sum_{i=i_1+1}^{\infty}\int_{c^{-i}}^{c^{-i+1}} \Prob\left(|X^{(1)}(t)|<\frac{a}{\sqrt{p+1}},\|X^{(j)}(t)\|<\frac{a}{\sqrt{p+1}},2\leq j\leq p\right)dt\\
	&\geq \sum_{i=i_1+1}^{\infty}\int_{c^{-i}}^{c^{-i+1}}\Prob\left(|X^{(1)}(c^i t)|<\frac{c^{\frac{i}{\alpha_1}-i_0}}{\sqrt{p+1}},\|X^{(j)}(c^i t)\|<K^{-1}\frac{c^{\frac{i}{\alpha_j'}-i_0}}{\sqrt{p+1}},2\leq j\leq p\right)dt\\
	&\geq\sum_{i=i_1+1}^{\infty}c^{-i}\int_1^c\Prob\left(
	\begin{array}{ll}
	|X^{(1)}(m)|<\frac{c^{\frac{i}{\alpha_1}-i_0}}{\sqrt{p+1}} \text{ and }\\
	\|X^{(j)}(m)\|< K^{-1}\frac{c^{\frac{i}{\alpha_j'}-i_0}}{\sqrt{p+1}}, 2\leq j\leq p
	\end{array}\right)dm.
	\end{align*}
	By Lemma \ref{denspos} choose $K_{10}>0$, $r>0$ and uniformly bounded Borel sets $J_m\subseteq\rr^{d-1}$ with Lebesgue measure $0<K_9\leq \lambda^{d-1}(J_m)<\infty$ for every $m\in[1,c)$ such that the bounded continuous density $g_m(x_1,\ldots,x_p)$ of $X(m)=X^{(m)}+\ldots+X^{(p)}(m)$ fulfills
	\begin{align*}
	g_m(x_1,\ldots,x_p)\geq K_{10} >0 \quad\text{ for all }\quad (x_1,\ldots,x_p)\in[-r,r]\times J_m
	\end{align*}
	and for every $m\in[1,c)$. Since $\{J_m\}_{m\in[1,c)}$ is uniformly bounded by Lemma 2.4 we are able to choose $0<\delta\leq c^{-3}<1$ such that
	$$
	\bigcup_{m\in[1,c)}J_m \subseteq \left\{\|x_j\|\leq\frac{K^{-1}c^{\frac{-\alpha_1}{\alpha_p}}}{\delta\sqrt{p+1}}, 2\leq j\leq p\right\}.
	$$
	Let $\eta=c^{\frac{2}{\alpha_p}}/\left(r\sqrt{p+1}\right)$.
	
	Since $\alpha_1>\alpha_2'>1$ there exists a constant $a_0\in(0,1]$ such that $(\eta a)^{\alpha_1}<(\delta a)^{\alpha_2'}$ for all $0<a\leq a_0$. Now, choose $i_2, i_3\in\nat_0$ such that $c^{-i_2}<\left(\delta c^{-i_0+1}\right)^{\alpha_2'}\leq c^{-i_2+1}$ and $c^{-i_3}<\left(\eta c^{-i_0}\right)^{\alpha_1}\leq c^{-i_3+1}$. Note that
	\begin{align*}
	c^{-i_3}<\left(\eta c^{-i_0}\right)^{\alpha_1}<(\delta a)^{\alpha_2'}\leq \left(\delta c^{-i_0+1}\right)^{\alpha_2'} \leq c^{-i_2+1}
	\end{align*}
	and
	\begin{align*}
	&c^{-(i_1+1)}=c^{-2}\cdot c^{-i_1+1}\geq c^{-2}\cdot c^{-i_0\alpha_2} \geq \left(c^{-2}\cdot c^{-i_0}\right)^{\alpha_2}\geq \left(c^{-2}\cdot c^{-i_0}\right)^{\alpha_2'}\\ 
	&= \left(c^{-3}\cdot c^{-i_0+1}\right)^{\alpha_2'} \geq \left(\delta c^{-i_0+1}\right)^{\alpha_2'}> c^{-i_2},
	\end{align*}
	hence $i_3\geq i_2-1$ and $i_1+1\leq i_2$. We further have for all $i=i_2,\ldots,i_3+1$ and every $j=2,\ldots,p$
	\begin{equation}\label{up}
		\frac{c^{i/\alpha_1-i_{0}}}{\sqrt{p+1}}\leq \frac{c^{(i_3+1)/\alpha_1-i_{0}}}{\sqrt{p+1}}\leq\frac{c^{2/\alpha_1}(\eta c^{-i_{0}})^{-1}c^{-i_{0}}}{\sqrt{p+1}}=\frac{c^{2/\alpha_1}}{\eta\sqrt{p+1}}=r
	\end{equation}
	and, since $\alpha_2'\geq\alpha_j'$ for $j=2,\ldots,p$,
	\begin{equation}\begin{split}\label{down}
		\frac{c^{i/\alpha_j'-i_{0}}}{\sqrt{p+1}} & \geq \frac{c^{i_2/\alpha_j'-i_{0}}}{\sqrt{p+1}}\geq \frac{(\delta c^{-i_0+1})^{-\alpha_2'/\alpha_j'}c^{-i_0}}{\sqrt{p+1}}\\
		& =\frac{(\delta^{-1}c^{i_0-1})^{\alpha_2'/\alpha_j'}c^{-i_0}}{\sqrt{p+1}}\geq\frac{c^{-\alpha_2'/\alpha_j'}}{\delta\sqrt{p+1}} \geq\frac{c^{-\alpha_1/\alpha_p}}{\delta\sqrt{p+1}}.
	\end{split}\end{equation}
	Let $I_m=(-\frac{c^{i/\alpha_1-i_{0}}}{\sqrt{p+1}},\frac{c^{i/\alpha_1-i_{0}}}{\sqrt{p+1}})\times J_m$ then together with the calculations above, we get using \eqref{up} and \eqref{down}
	\begin{align*}
		\Exp[T(a,s)] & \geq \sum_{i=i_{2}}^{i_3+1}c^{-i}\int_{1}^{c}P\left(\begin{array}{c} |X^{(1)}(m)|<\frac{c^{i/\alpha_1-i_{0}}}{\sqrt{p+1}}\text{ and}\\ \|X^{(j)}(m)\|\leq K^{-1}\frac{c^{i/\alpha_j'-i_{0}}}{\sqrt{p+1}}, 2\leq j\leq p\end{array}\right)dm\\
		&\geq \sum_{i=i_{2}}^{i_{3}+1}c^{-i}\int_{1}^{c}\int_{I_m} g_{m}(x)\,dx\,dm
		\geq \sum_{i=i_{2}}^{i_{3}+1}c^{-i}(c-1)\,2\,\frac{c^{i/\alpha_1-i_{0}}}{\sqrt{p+1}}\cdot K_{10}\cdot K_9\\
		&= Kc^{-i_{0}}\sum_{i=i_{2}}^{i_{3}+1}\left(c^{-i}\right)^{1-\frac{1}{\alpha_{1}}}
		= Kc^{-i_{0}}\left(\frac{1 - \left(c^{-(i_{3}+2)}\right)^{1-\frac{1}{\alpha_{1}}}} {1-c^{\frac{1}{\alpha_{1}}-1}} - \frac{1-\left(c^{-i_{2}}\right)^{1-\frac{1}{\alpha_{1}}}} {1-c^{\frac{1}{\alpha_{1}}-1}}\right)\\
		& = Kc^{-i_{0}}\left(\left(c^{-i_{2}}\right)^{1-\frac{1}{\alpha_{1}}} -\left(c^{-(i_{3}+2)}\right)^{1-\frac{1}{\alpha_{1}}}\right)\\	
		& \geq K_{41}\left(c^{-i_{0}}\right)^{1+\alpha_2'\left(1-\frac{1}{\alpha_1}\right)} - K_{42}\left(c^{-i_{0}}\right)^{\alpha_{1}}.
	\end{align*}
	Since $1+\alpha_{2}'(1-\frac{1}{\alpha_{1}}) < 1+\alpha_{1}(1-\frac{1}{\alpha_{1}})=\alpha_{1}$ we have
	$\left(c^{-i_{0}}\right)^{\alpha_{1}-\left(1+\alpha_2'\left(1-\frac{1}{\alpha_1}\right)\right)}\to0$ if $a\to0$, i.e. $i_{0}\to\infty$. Hence we can further choose $a_{0}$ sufficiently small, such that	
	$$
	\Exp[T_Z(a,s)] \geq K_4 a^{1+\alpha_2'(1-\frac{1}{\alpha_1})}
	$$
	for all $0<a\leq a_0$.
	
	(iv) Let $0<\alpha_j<\alpha_j'<\alpha_{j-1}$ for all $j=2,\ldots,p$, and 	additionally, let $\alpha_2<\alpha_2'<1$. Now choose $i_0, i_1\in\nat_0$ such that $c^{-i_0}<a\leq c^{-i_0+1}$ and $c^{-i_1}<\frac{a}{\sqrt{p+1}}\leq c^{-i_1+1}$. For $t\in(0,1]$ we can write $t=mc^{-i}$ with $m\in[1,c)$ and $i\in\nat_0$. By Lemma \ref{specbound} we then have
	\begin{equation}\begin{split}\label{ubXj}
	\|X^{(j)}(t)\| \stackrel{\rm d}{=} \|c^{-iE_j}X^{(j)}(m)\|\leq \|c^{-iE_j}\|\,\|X^{(j)}(m)\| \leq K\,c^{-i/{\alpha_{j}'}}\|X^{(j)}(c^it)\|
	\end{split}\end{equation}
	for all $j=1,\ldots,p$. Note that, since $d_1=1$, for $j=1$ in \eqref{ubXj} we can choose $K=1$ and $\alpha_1'=\alpha_1$. Altogether this gives us
	\begin{align*}
	\Exp[T_Z(a,s)] & = \int_0^s \Prob\left(\|Z(t)\|<a\right)dt = \int_0^s \Prob\left(\|(t,X(t))\|<a\right)dt\\
	&\geq\int_0^s \Prob\left(|X^{(1)}(t)|<\frac{a}{\sqrt{p+1}}, \|X^{(j)}(t)\|<\frac{a}{\sqrt{p+1}},2\leq j\leq p, |t|<\frac{a}{\sqrt{p+1}}\right)dt\\
	&= \int_0^{\frac{a}{\sqrt{p+1}}} \Prob\left(|X^{(1)}(t)|<\frac{a}{\sqrt{p+1}},\|X^{(j)}(t)\|<\frac{a}{\sqrt{p+1}},2\leq j\leq p\right)dt\\
	&\geq \int_0^{c^{-i_1}} \Prob\left(|X^{(1)}(t)|<\frac{a}{\sqrt{p+1}},\|X^{(j)}(t)\|<\frac{a}{\sqrt{p+1}},2\leq j\leq p\right)dt\\
	& = \sum_{i=i_1+1}^{\infty}\int_{c^{-i}}^{c^{-i+1}} \Prob\left(|X^{(1)}(t)|<\frac{a}{\sqrt{p+1}},\|X^{(j)}(t)\|<\frac{a}{\sqrt{p+1}},2\leq j\leq p\right)dt\\
	&\geq \sum_{i=i_1+1}^{\infty}\int_{c^{-i}}^{c^{-i+1}}\Prob\left(|X^{(1)}(c^i t)|<\frac{c^{\frac{i}{\alpha_1}-i_0}}{\sqrt{p+1}},\|X^{(j)}(c^i t)\|<K^{-1}\frac{c^{\frac{i}{\alpha_j'}-i_0}}{\sqrt{p+1}},2\leq j\leq p\right)dt\\
	&\geq\sum_{i=i_1+1}^{\infty}c^{-i}\int_1^c\Prob\left(
	\begin{array}{ll}
	|X^{(1)}(m)|<\frac{c^{\frac{i}{\alpha_1}-i_0}}{\sqrt{p+1}} \text{ and }\\
	\|X^{(j)}(m)\|< K^{-1}\frac{c^{\frac{i}{\alpha_j'}-i_0}}{\sqrt{p+1}}, 2\leq j\leq p
	\end{array}\right)dm.
	\end{align*}
	Analogously to the reasoning above, by Lemma \ref{denspos} choose $K_{10}>0$, $r>0$ and uniformly bounded Borel sets $J_m\subseteq\rr^{d-1}$ with Lebesgue measure $0<K_9\leq \lambda^{d-1}(J_m)<\infty$ for every $m\in[1,c)$ such that the bounded continuous density $g_m(x_1,\ldots,x_p)$ of $X(m)=X^{(m)}+\ldots+X^{(p)}(m)$ fulfills
	\begin{align*}
	g_m(x_1,\ldots,x_p)\geq K_{10} >0 \quad\text{ for all }\quad (x_1,\ldots,x_p)\in[-r,r]\times J_m
	\end{align*}
	and for every $m\in[1,c)$. Since $\{J_m\}_{m\in[1,c)}$ is uniformly bounded by Lemma 2.4 we are now able to choose $0<\delta\leq(\sqrt{p+1}\cdot c^3)^{-1}<1$ such that
	$$
	\bigcup_{m\in[1,c)}J_m \subseteq \left\{\|x_j\|\leq\frac{K^{-1}c^{\frac{-\alpha_1}{\alpha_p}}}{\delta\sqrt{p+1}}, 2\leq j\leq p\right\}.
	$$
	Let $\eta=c^{\frac{2}{\alpha_p}}/\left(r\sqrt{p+1}\right)$.
	
	Since $\alpha_1>1$ there exists a constant $0<a_0\leq1$ such that we have $(\eta a)^{\alpha_1}<\delta a$ for all $0<a\leq a_0$. Now, choose $i_2, i_3\in\nat_0$ such that $c^{-i_2}<\delta c^{-i_0+1}\leq c^{-i_2+1}$ and $c^{-i_3}<\left(\eta c^{-i_0}\right)^{\alpha_1}\leq c^{-i_3+1}$. Note that
	\begin{align*}
	c^{-i_3}<\left(\eta c^{-i_0}\right)^{\alpha_1}<\left(\eta a\right)^{\alpha_1}<\delta a\leq \delta c^{-i_0+1}\leq c^{-i_2+1}
	\end{align*}
	and, since $\delta\leq\frac{1}{\sqrt{p+1}}\cdot c^{-3}$,
	\begin{align*}
	c^{-(i_1+1)}=c^{-2}\cdot c^{-(i_1-1)} \geq c^{-2}\cdot\frac{a}{\sqrt{p+1}} > c^{-2}\cdot\frac{c^{-i_0}}{\sqrt{p+1}} =\frac{c^{-3}}{\sqrt{p+1}}\cdot c^{-i_0+1}\geq \delta c^{-i_0+1}>c^{-i_2}.
	\end{align*}
	Hence, we also get $i_2-1\leq i_3$ and $i_1+1\leq i_2$. Analogously to the calculations above, we further have for all $i=i_2,\ldots,i_3+1$ that 
	\begin{equation}\label{up2}
	\frac{c^{i/\alpha_1-i_{0}}}{\sqrt{p+1}}\leq r
	\end{equation}
	and, since $\alpha_j'<1$ for all $j=2,\cdot p$,
	\begin{equation}\begin{split}\label{down2}
	\frac{c^{i/\alpha_j'-i_{0}}}{\sqrt{p+1}} & \geq \frac{c^{i_2/\alpha_j'-i_{0}}}{\sqrt{p+1}}\geq \frac{(\delta c^{-i_0+1})^{-1/\alpha_j'}c^{-i_0}}{\sqrt{p+1}} =\frac{(\delta^{-1}c^{i_0-1})^{1/\alpha_j'}c^{-i_0}} {\sqrt{p+1}}\\
	& \geq\frac{c^{-1/\alpha_j'}}{\delta\sqrt{p+1}} \geq\frac{c^{-1/\alpha_p}}{\delta\sqrt{p+1}}\geq \frac{c^{-\alpha_1/\alpha_p}}{\delta\sqrt{p+1}}.
	\end{split}\end{equation}
	Define the subsets $\{I_m:m\in[1,c)\}\subseteq \rd$ as above. Similarly to the calculations above, using \eqref{up2} and \eqref{down2} we arrive at
	\begin{equation}
	\Exp[T_Z(a,s)] \geq Kc^{-i_{0}}\left(\left(c^{-i_{2}}\right)^{1-\frac{1}{\alpha_{1}}} -\left(c^{-(i_{3}+2)}\right)^{1-\frac{1}{\alpha_{1}}}\right)
	\end{equation}	
	Altogether, we get
	\begin{align*}
	\Exp[T_Z(a,s)]& 
	\geq  Kc^{-i_{0}}\left(\left(c^{-i_{2}}\right)^{1-\frac{1}{\alpha_{1}}} -\left(c^{-(i_{3}+2)}\right)^{1-\frac{1}{\alpha_{1}}}\right)\\
	&\geq K_{51}c^{-i_{0}}\left(c^{-i_0}\right)^{1-\frac{1}{\alpha_{1}}} - K_{52}c^{-i_{0}}\left(c^{-i_0}\right)^{\alpha_1-1}\\
	&= K_{51}\left(c^{-i_0}\right)^{2-\frac{1}{\alpha_{1}}}-K_{52}\left(c^{-i_0}\right)^{\alpha_1}.
	\end{align*}
	Since $\alpha_1>1$ and therefore $2-1/\alpha_1<1+(1-1/\alpha_1)<1+\alpha_1(1-1/\alpha_1)=\alpha_1$, similarly to the above, we can choose $a_0$ sufficiently small, such that
	\begin{align*}
	\Exp[T_Z(a,s)]\geq K_5 a^{2-\frac{1}{\alpha_1}}.
	\end{align*}
	for all $0<a\leq a_0$.
\end{proof}

Similarly to the proof of Lemma 3.4 in \cite{KW}, we can now find a suitable covering of $Z(B)$ and prove the desired upper bounds.

\begin{lemma}\label{upperbounds}
	Let $X=\{X(t), t\in\rr_+\}$ be an operator semistable L\'evy process on $\rd$ with $d\geq2$. Then for any Borel set $B\subseteq\rr_+$ we have almost surely
	\begin{align*}
	\dim_H Gr_X(B)\leq \left\{
	\begin{array}{ll} \alpha_1\dim_H B, & \text{ if } \alpha_1\dim_H B \leq d_1,\alpha_1\geq 1,\hfill (i) \\
	\dim_H B, & \text{ if } \alpha_1\dim_H B \leq d_1, \alpha_1< 1, \hfill (ii)\\
	1+\alpha_2(\dim_H B - \frac{1}{\alpha_1}), & \text{ if } \alpha_1\dim_H B > d_1, \alpha_1>\alpha_2\geq1,\hfill (iii)\\
	1+\dim_HB-\frac{1}{\alpha_1}, & \text{ if } \alpha_1\dim_H B > d_1, \alpha_1>1>\alpha_2.\hfill(iv)
	\end{array}
	\right.
	\end{align*}
\end{lemma}

\begin{proof}
(i)  Assume $\alpha_1\dim_H B\leq d_1$ and $\alpha_1\geq 1$. First, we consider the case where $\alpha_1\leq d_1$. For $\gamma>\dim_{H}B$ choose $\alpha_{1}'>\alpha_{1}$ such that $\gamma'=1-\frac{\alpha_{1}'}{\alpha_{1}}+\gamma>\dim_{H}B$. For any $\varepsilon\in(0,1]$, by definition of the Hausdorff dimension,  there now exists a sequence $\{I_{i}\}_{i\in\mathbb{N}}$ of intervals in $\rr_{+}$ of length $|I_{i}|<\varepsilon$ such that
$$B\subseteq \bigcup_{i=1}^{\infty}I_{i}\quad\text{ and }\quad\sum_{i=1}^{\infty}|I_{i}|^{\gamma'}<1.$$
Let $s_{i}:=|I_{i}|$ und $b_{i}:=|I_{i}|^{\frac{1}{\alpha_{1}}}$ then $(b_i/3)^{\alpha_1}<s_i$. It follows by Lemma \ref{PT} and Theorem \ref{ST} that $Z(I_{i})$ can be covered by $M_{i}$ cubes $C_{ij}\in \Lambda(b_{i})$ of side $b_{i}$ such that for every $i\in\nat$ we have
\begin{align*}
\Exp[M_{i}]\leq 2K_{1}s_{i}\left(\Exp\left[T_Z\left(\tfrac{b_{i}}{3},s_{i}\right)\right]\right)^{-1}\leq 2K_{1}s_{i}K_{2}^{-1}\left(\tfrac{b_{i}}{3}\right)^{-\alpha_{1}'}= K \,s_{i}b_{i}^{-\alpha_{1}'} = K \,|I_{i}|^{1-\frac{\alpha_{1}'}{\alpha_{1}}}.
\end{align*}
Note that $Z(B)\subseteq\bigcup_{i=1}^{\infty} \bigcup_{j=1}^{M_{i}}C_{ij}$, where $b_{i}\sqrt{d+1}$ is the diameter of $C_{ij}$. In other words, $\{C_{ij}\}$ is a $(\varepsilon^{1/\alpha_{1}}\sqrt{d+1})$-covering of $X(B)$. By monotone convergence we have
\begin{align*}
\Exp\left[\sum_{i=1}^{\infty}M_{i}b_{i}^{\alpha_{1}\gamma}\right] & = \sum_{i=1}^{\infty}\Exp\left[M_{i}b_{i}^{\alpha_{1}\gamma}\right]\leq \sum_{i=1}^{\infty} K \,|I_{i}|^{1-\frac{\alpha_{1}'}{\alpha_{1}}}\, |I_{i}|^{\gamma}= K\sum_{i=1}^{\infty}|I_{i}|^{\gamma'} \leq K.
\end{align*}
Letting $\varepsilon\to0$, i.e $b_{i}\to0$ and applying Fatou's lemma we get
\begin{align*}
& \Exp\left[\mathcal{H}^{\alpha_{1}\gamma}(X(B))\right] \leq \Exp\left[\liminf_{\varepsilon\to0} \sum_{i=1}^{\infty}\sum_{j=1}^{M_{i}}\left(b_{i}\sqrt{d+1}\right)^{\alpha_{1}\gamma}\right]\\
& \leq \liminf_{\varepsilon\to 0} \sqrt{d+1}^{\;\alpha_{1}\gamma}\Exp\left[\sum_{i=1}^{\infty}M_{i}b_{i}^{\alpha_{1}\gamma}\right]\leq\sqrt{d+1}^{\;\alpha_{1}\gamma} K < \infty,
\end{align*}
which shows that $\dim_{H}Z(B)\leq \alpha_{1}\gamma$ almost surely. And since $\gamma>\dim_{H}B$ is arbitrary, we get $\dim_{H}Z(B) \leq \alpha_{1}\dim_{H}B$ almost surely.

Now, assume that $\alpha_1\dim_H B\leq d_1$ and $\alpha_1\geq 1$ and $\alpha_1>d_1$. To be able to argue in the same way as before, we have to show that in case $\alpha_1>d_1$ the same lower bound $\Exp\left[T_Z\left(a,s\right)\right]\geq Ka^{\alpha_1'}$ holds for the expected sojourn time. By Theorem \ref{ST} we have for $1\leq\alpha_2<\alpha_2'<\alpha_1<\alpha_1'$
\begin{align*}
	\Exp\left[T_Z\left(a,s\right)\right]\geq K a^{1+\alpha_2'(1-\frac{1}{\alpha_1})}\geq K a^{1+\alpha_1(1-\frac{1}{\alpha_1})} = K a^{\alpha_1}\geq K a^{\alpha_1'}
\end{align*}
and for $0<\alpha_2'<\alpha_1<\alpha_1'$
\begin{align*}
\Exp\left[T_Z\left(a,s\right)\right]\geq K a^{2-\frac{1}{\alpha_1}} = K a^{1+1-\frac{1}{\alpha_1}}\geq K a^{1+\alpha_1(1-\frac{1}{\alpha_1})} \geq K a^{\alpha_1'}.
\end{align*}
Altogether, we get the desired lower bound for all $0<a\leq 1$ small enough and $a^{\alpha_1}\leq s\leq1$. Hence, as above the same conclusion $\dim_H Z(B)\leq \alpha_1\dim_H B$ holds almost surely.

(ii) Assume $\alpha_1\dim_H B\leq d_1$ and $\alpha_1<1\leq d_1$. For $\gamma>\dim_H B$, choose $\beta>1$ such that $\gamma'=1-\beta+\gamma>\dim_H B$. For $\varepsilon\in(0,1]$, define the sequence $\{I_i\}_{i\in\nat}$ of intervals as in part (i). Let $s_i=b_i:=|I_i|$; then $b_i/(3\sqrt{p+1})<s_i$. Again by Lemma \ref{PT} and Theorem \ref{ST} it follows that $Z(I_{i})$ can be covered by $M_{i}$ cubes $C_{ij}\in \Lambda(b_{i})$ of side $b_{i}$ such that for every $i\in\nat$ we have
\begin{align*}
\Exp[M_{i}]\leq 2K_{1}s_{i}\left(\Exp\left[T_Z\left(\tfrac{b_{i}}{3},s_{i}\right)\right]\right)^{-1}\leq 2K_{1}s_{i}K_{3}^{-1}\left(\tfrac{b_{i}}{3}\right)^{-\beta}= K \,s_{i}b_{i}^{-\beta} = K \,|I_{i}|^{1-\beta}.
\end{align*}
By monotone convergence we have
\begin{align*}
\Exp\left[\sum_{i=1}^{\infty}M_{i}b_{i}^{\gamma}\right] & = \sum_{i=1}^{\infty}\Exp\left[M_{i}b_{i}^{\gamma}\right]\leq \sum_{i=1}^{\infty} K \,|I_{i}|^{1-\beta}\, |I_{i}|^{\gamma}= K\sum_{i=1}^{\infty}|I_{i}|^{\gamma'} \leq K.
\end{align*}
Since $\gamma>\dim_H B$ and $\beta>1$ are arbitrary, with the same arguments as above we get $\dim_H Z(B)\leq \dim_H B$ almost surely.

(iii) Assume $\alpha_1\dim_H B>d_1$ and $\alpha_2\geq1$. Since $\dim_H B\leq 1$, we have $\alpha_1>d_1=1$. For $\gamma>\dim_{H}B$ choose $\alpha_{2}'>\alpha_{2}$ such that $\gamma'=1-\frac{\alpha_{2}'}{\alpha_{2}}+\frac{\alpha_{2}'}{\alpha_{2}}\gamma>\dim_{H}B$. For $\varepsilon\in(0,1]$ define $\{I_{i}\}_{i\in\mathbb{N}}$ as in part (i) and let $s_{i}:=|I_{i}|$ and $b_{i}:=|I_{i}|^{\frac{1}{\alpha_{2}}}$. Then $(b_i/3)^{\alpha_2}<s_i$. Again, by Lemma \ref{PT} and Theorem \ref{ST} it follows that $Z(I_{i})$ can be covered by $M_{i}$ cubes $C_{ij}\in \Lambda(b_{i})$ of side $b_{i}$ such that for every $i\in\nat$ we have
\begin{align*}
\Exp[M_{i}]&\leq 2K_{1}s_{i}\left(\Exp\left[T_Z\left(\tfrac{b_{i}}{3},s_{i}\right)\right]\right)^{-1}\leq 2K_{1}s_{i}K_{4}^{-1}\left(\tfrac{b_{i}}{3}\right)^{-1-\alpha_2'(1-\frac{1}{\alpha_1})}\\ &= K \,s_{i}b_{i}^{-1-\alpha_2'(1-\frac{1}{\alpha_1})} = K \,|I_{i}|^{1-\frac{1}{\alpha_2}-\frac{\alpha_2'}{\alpha_2}\cdot(1-\frac{1}{\alpha_1})}.
\end{align*}
By monotone convergence we have
\begin{align*}
\Exp\left[\sum_{i=1}^{\infty}M_{i}b_{i}^{1+\alpha_{2}'(\gamma-\frac1{\alpha_1})}\right] & \leq \sum_{i=1}^{\infty} K \,|I_{i}|^{1-\frac{1}{\alpha_2}-\frac{\alpha_2'}{\alpha_2}\cdot(1-\frac{1}{\alpha_1})}\, |I_{i}|^{\frac1{\alpha_2}+\frac{\alpha_{2}'}{\alpha_2}(\gamma-\frac1{\alpha_1})}= K\sum_{i=1}^{\infty}|I_{i}|^{\gamma'} \leq K.
\end{align*}
Since $\gamma>\dim_{H}B$ and $\alpha_2'>\alpha_2$ are arbitrary, with the same arguments as in part (i) we get $\dim_{H}Z(B) \leq 1+\alpha_2(\dim_{H}B-\frac1{\alpha_1})$ almost surely.

(iv) Assume $\alpha_1\dim_H B>d_1$ and $\alpha_2<1$. Since $\dim_H B\leq 1$, we have $\alpha_1>d_1=1$. Let $\gamma=\gamma'>\dim_{H}B$. For $\varepsilon\in(0,1]$ define $\{I_{i}\}_{i\in\mathbb{N}}$ as in part (i) and let $s_{i}:=|I_{i}|$ and $b_{i}:=|I_{i}|$. Then $b_i/(3\sqrt{p+1})<s_i$. Again, by Lemma \ref{PT} and Theorem \ref{ST} it follows that $Z(I_{i})$ can be covered by $M_{i}$ cubes $C_{ij}\in \Lambda(b_{i})$ of side $b_{i}$ such that for every $i\in\nat$ we have
\begin{align*}
\Exp[M_{i}]&\leq 2K_{1}s_{i}\left(\Exp\left[T_Z\left(\tfrac{b_{i}}{3},s_{i}\right)\right]\right)^{-1}\leq 2K_{1}s_{i}K_{5}^{-1}\left(\tfrac{b_{i}}{3}\right)^{-2+\frac{1}{\alpha_1}}= K \,s_{i}b_{i}^{-2+\frac{1}{\alpha_1}} = K \,|I_{i}|^{-1+\frac{1}{\alpha_1}}.
\end{align*}
By monotone convergence we have
\begin{align*}
\Exp\left[\sum_{i=1}^{\infty}M_{i}b_{i}^{1+\gamma-\frac1{\alpha_1}}\right] & \leq \sum_{i=1}^{\infty} K \,|I_{i}|^{-1+\frac{1}{\alpha_1}}\, |I_{i}|^{1+\gamma-\frac1{\alpha_1}}=K\sum_{i=1}^{\infty}|I_{i}|^{\gamma}= K\sum_{i=1}^{\infty}|I_{i}|^{\gamma'} \leq K.
\end{align*}
Since $\gamma>\dim_{H}B$ is arbitrary, we get $\dim_{H}Z(B) \leq 1+\dim_{H}B-\frac1{\alpha_1}$ almost surely.
\end{proof}


\subsection{Lower Bounds}

In order to obtain the lower bounds of $\dim_H Gr_X(B)$ we apply Frostman's Lemma and Theorem and use the relationship between the Hausdorff dimension and the capacitary dimension (see \cite{Fal, Mat} for details).

\begin{lemma}\label{lowerbounds}
	Let $X=\{X(t), t\in\rr_+\}$ be an operator semistable L\'evy process on $\rd$ with $d\geq2$. Then for any Borel set $B\subseteq\rr_+$ we have almost surely
	\begin{align*}
	\dim_H Gr_X(B)\geq \left\{
	\begin{array}{ll} \alpha_1\dim_H B, & \text{ if } \alpha_1\dim_H B \leq d_1,\alpha_1\geq 1,\hfill (i) \\
	\dim_H B, & \text{ if } \alpha_1\dim_H B \leq d_1, \alpha_1< 1, \hfill (ii)\\
	1+\alpha_2(\dim_H B - \frac{1}{\alpha_1}), & \text{ if } \alpha_1\dim_H B > d_1, \alpha_1>\alpha_2\geq1,\hfill (iii)\\
	1+\dim_HB-\frac{1}{\alpha_1}, & \text{ if } \alpha_1\dim_H B > d_1, \alpha_1>1>\alpha_2.\hfill(iv)
	\end{array}
	\right.
	\end{align*}
\end{lemma}

\begin{proof}
	(i)+(iii) Since projections are Lipschitz continuous, we have $\dim_H Gr_X(B)\geq \dim_H X(B)$. Hence, the desired lower bounds in these two parts can be deduced from the dimension result \eqref{DimensionRange} for the range of an operator semistable process.
	
	(ii) Choose $0<\gamma<\dim_HB\leq1$. Then by Frostman's lemma there exists a probability measure $\sigma$ on $B$ such that
	\begin{equation}\label{sigma}
	\int_B\int_B \frac{\sigma(ds)\sigma(dt)}{|s-t|^{\gamma}}<\infty.
	\end{equation}
	In order to prove $\dim_H Gr_X(B) = \dim_H Z(B)\geq\gamma$ almost surely, by Frostman's theorem \cite{Kah, Mat} it suffices to show that
	\begin{equation}\label{Xsigma}
	\int_{B}\int_{B}\Exp\left[\|Z(s)-Z(t)\|^{-\gamma}\right]\,\sigma(ds)\,\sigma(dt)<\infty.
	\end{equation}
	Let
	$$
	K_{11} = \sup_{m\in[1,c)}\Exp\left[\|X^{(1)}(m)\|^{-\gamma}\right]<\infty
	$$
	by Lemma \ref{uniexp}, since $\gamma<1\leq d_1$. In order to verify \eqref{Xsigma} we split the domain of integration into two parts.
	\begin{itemize}
		\item[(a)] Assume $|s-t|\leq1$. Then
		\begin{align*}
		\Exp\left[\left\|\begin{pmatrix} t\\X(t) \end{pmatrix}-\begin{pmatrix} s\\X(s) \end{pmatrix}\right\|^{-\gamma}\right]\leq \Exp\left[|s-t|^{-\gamma}\right]=|s-t|^{-\gamma}.
		\end{align*}
		\item[(b)] Now assume $|s-t|\geq 1$ and choose $\alpha_{1}'>\alpha_{1}$. Write $|s-t|=mc^{i}$ with $m\in[1,c)$ and $i\in\mathbb{N}_{0}$. Using again Lemma \ref{specbound} we get
		\begin{align*}
		&\Exp\left[\left\|\begin{pmatrix} t\\X(t) \end{pmatrix}-\begin{pmatrix} s\\X(s) \end{pmatrix}\right\|^{-\gamma}\right]\leq\Exp\left[\|X^{(1)}(t)-X^{(1)}(s)\|^{-\gamma}\right]\\
		& = \|c^{-iE_{1}}\|^{\gamma} \Exp\left[\|X^{(1)}(m)\|^{-\gamma}\right] \leq K\,c^{-\gamma i/\alpha_{1}'} K_{11}\leq K\,K_{11} = K_{12}.
		\end{align*}
	\end{itemize}
	Combining part (a) and (b) in \eqref{Xsigma} and applying \eqref{sigma} we get the desired upper bound.
	
	(iv) Assume $\alpha_1\dim_HB>d_1$ then $\alpha_1>d_1=1$. Choose $1<\gamma<1+\dim_HB-\frac1{\alpha_1}$, then $\rho=\gamma-1+\frac1{\alpha_1}<\dim_HB$. By Frostman's lemma, there exists again a probability measure $\sigma$ on $B$ such that
	$$
	\int_B\int_B \frac{\sigma(ds)\sigma(dt)}{|s-t|^{\rho}}<\infty.
	$$
	Again, in order to verify \eqref{Xsigma} we split the domain of integration into two parts.
	
	First assume that $|s-t|=mc^{-i}\leq1$ with $m\in[1,c)$ and $i\in\nat_0$. Since $d_1=1$ we get
	\begin{align*}
	&\Exp\left[\left\|\begin{pmatrix} t\\X(t) \end{pmatrix}-\begin{pmatrix} s\\X(s) \end{pmatrix}\right\|^{-\gamma}\right]\leq\Exp\left[\left(c^{-i\frac{2}{\alpha_1}}\cdot|X^{(1)}(m)|+|s-t|^2\right)^{-\frac\gamma2}\right]\\
	&\leq K\int_\rr \;\frac1{c^{-i\frac{\gamma}{\alpha_1}}\cdot|x_1|^\gamma+|s-t|^\gamma}\cdot g_m(x_1)dx_1\\
	&=K\int_\rr \;\frac1{m^{-\frac{\gamma}{\alpha_1}}\left(mc^{-i}\right)^{\frac{\gamma}{\alpha_1}}\cdot|x_1|^\gamma+|s-t|^\gamma}\cdot g_m(x_1)dx_1\\
	&\leq K\int_\rr \;\frac1{c^{-\frac{\gamma}{\alpha_1}}\cdot|s-t|^{\frac{\gamma}{\alpha_1}}|x_1|^\gamma+|s-t|^\gamma}\cdot g_m(x_1)dx_1\\
	&\leq K\int_\rr \;\frac1{|s-t|^{\frac{\gamma}{\alpha_1}}|x_1|^\gamma+|s-t|^\gamma}\cdot g_m(x_1)dx_1\\
	&=K\cdot|s-t|^{-\frac{\gamma}{\alpha_1}}\int_\rr \;\frac1{|x_1|^\gamma+|s-t|^{\gamma(1-\frac1{\alpha_1})}}\cdot g_m(x_1)dx_1 =: K\cdot|s-t|^{-\frac{\gamma}{\alpha_1}} \cdot I_m,
	\end{align*}
	where $g_m(x_1)$ is the density function of $X^{(1)}(m)$. Let
	$$
	F_m(r_1)=\Prob\left(|X^{(1)}(m)|\leq r_1\right)=\int_{|x_1|\leq r_1}g_m(x_1)dx_1
	$$
	and note that by Lemma \ref{uniconfun} 
	$$
	\sup_{m\in[1,c)}\sup_{x_1\in\rr}|g_m(x_1)|\leq K_8<\infty.
	$$
	This leads to
	$$
	F_m(r_1)\leq 1 \wedge 2K_8\cdot r_1 \quad \forall r_1\geq0 \text{ and } \forall m\in[1,c).
	$$
	We denote $z=|s-t|^{1-\frac1{\alpha_1}}$. By using integration by parts, we deduce
	\begin{align*}
	I_m &= \int_0^\infty\! \frac1{r_1^{\gamma}+z^{\gamma}}\;F_m(dr_1)\\
	&= \left[\frac1{r_1^{\gamma}+z^{\gamma}}\; F_m(r_1)\right]_0^\infty + \int_0^\infty\! \frac{\gamma r_1^{\gamma-1}}{\left(r_1^{\gamma}+z^{\gamma}\right)^2}\;F_m(r_1)dr_1\\
	&\leq K \int_0^\infty\! \frac{\gamma r_1^{\gamma-1}}{\left(r_1^{\gamma}+z^{\gamma}\right)^2}\;r_1 dr_1 = K \int_0^\infty\! \frac{\gamma r_1^{\gamma}}{\left(r_1^{\gamma}+z^{\gamma}\right)^2}\;dr_1\\
	&= K \int_0^\infty\! \frac{z\gamma\cdot (zs_1)^{\gamma}}{\left((zs_1)^{\gamma}+c^{\gamma}\right)^2}\;ds_1\\
	&= K z^{-(\gamma-1)}\cdot\int_0^\infty\!\frac{\gamma s_1^{\gamma}}{\left(s_1^{\gamma}+1\right)^2}\;ds_1\\
	&\leq K z^{-(\gamma-1)} = K\;|s-t|^{-(\gamma-1)(1-\frac1{\alpha_1})},
	\end{align*}
	where the last integral is finite since $\gamma>1$. Together we get for $|s-t|\leq1$
	\begin{align*}
	&\Exp\left[\left\|\begin{pmatrix} t\\X(t) \end{pmatrix}-\begin{pmatrix} s\\X(s) \end{pmatrix}\right\|^{-\gamma}\right]\leq K\;|s-t|^{-\gamma+1-\frac1{\alpha_1}}= K\;|s-t|^{-\rho}.
	\end{align*}
	For $|s-t|=mc^i\geq 1$ with $m\in[1,c)$ and $i\in\mathbb{N}_{0}$ choose $\alpha_{2}'>\alpha_{2}$. Then by Lemma \ref{specbound} we have
	\begin{align*}
	&\sup_{|s-t|\geq1}\Exp\left[\left\|\begin{pmatrix} t\\X(t) \end{pmatrix}-\begin{pmatrix} s\\X(s) \end{pmatrix}\right\|^{-\gamma}\right]\leq\sup_{|s-t|\geq1}\Exp\left[\|X(t)-X(s)\|^{-\gamma}\right]\\
	&\leq \sup_{m\in[1,c)}\; \sup_{i\in\nat_0}\; \Exp\left[\|X(mc^{i})\|^{-\gamma}\right]\\
	& \leq \sup_{m\in[1,c)}\; \sup_{i\in\nat_0}\;\Exp\left[\left(c^{i\frac{2}{\alpha_{1}}}|X^{(1)}(m)|^{2} + c^{i\frac{2}{\alpha_{2}'}}\|X^{(2)}(m)\|^{2} \right)^{-\frac{\gamma}{2}}\right]\\
	& \leq \sup_{m\in[1,c)}\; \sup_{i\in\nat_0}\;\Exp\left[\|(X^{(1)}(m),X^{(2)}(m))\|^{-\gamma}\right]\\
	&= \sup_{m\in[1,c)}\; \Exp\left[\|(X^{(1)}(m),X^{(2)}(m))\|^{-\gamma}\right]\leq K_{16}<\infty
	\end{align*}
	uniformly in $m\in[1,c)$ in view of Lemma \ref{uniexp}, since $\gamma<2\leq 1+d_{2}$. Therefore it follows from the calculations above that
	$$
	\int_B\int_B\Exp\left[\left\|\begin{pmatrix} t\\X(t) \end{pmatrix}-\begin{pmatrix} s\\X(s) \end{pmatrix}\right\|^{-\gamma}\right]\sigma(ds)\sigma(dt)<\infty.
	$$
	Using Frostman's theorem we have
	$$
	\dim_HGr_X(E)\geq\gamma.
	$$
	Since $\gamma<1+\dim_HB-\frac1{\alpha_1}$ was arbitrary this concludes the proof.
\end{proof}

\subsection{Proof of Main Results}
Theorem \ref{mainresult1} now follows directly from Lemma \ref{upperbounds} and Lemma \ref{lowerbounds}. It remains to prove the corresponding dimension result for the one-dimensional case as stated in Theorem \ref{mainresult2}. For $\alpha\dim_H B\leq 1$ Lemma \ref{upperbounds} and \ref{lowerbounds} are still valid for $d=1$ with $\alpha_1 := \alpha$. In case $\alpha\dim_H B > 1 = d$ the proof runs analogously to Lemma \ref{upperbounds} part (iv) and Lemma \ref{lowerbounds} part (iv).

\bibliographystyle{plain}

\end{document}